\documentclass[11pt]{article}
\usepackage{amsmath, amssymb, amscd, amsthm, amsfonts}
\usepackage{graphicx}
\usepackage{hyperref}

\title{$TFU$ extensions in LCA groups}
\author{Aliakbar Alijani}
\date{}

\usepackage[all]{xy}
\theoremstyle{plain}
 \newtheorem{theorem}{Theorem}[section]
 \newtheorem{lemma}[theorem]{Lemma}
\theoremstyle{definition}
\newtheorem{definition}[subsection]{Definition}
 \newtheorem{example}[theorem]{Example}
 \newtheorem{remark}[theorem]{Remark}
 \newtheorem{corollary}[theorem]{Corollary}

\begin{document}

\maketitle

\begin{abstract}
Let $\ell$ be the category of all locally compact abelian (LCA) groups. Let $G\in\ell$ and $H\subseteq G$. The first Ulm subgroup of $G$ is denoted by $G^{(1)}$ and the closure of $H$ by $\overline{H}$. A proper short exact sequence $0\to A\stackrel{\phi}{\to} B\stackrel{\psi}{\to} C\to 0$ in $\ell$ is said to be a $TFU$ extension if $0\to \overline{A^{(1)}}\stackrel{\overline{\phi}}{\to} \overline{B^{(1)}}\stackrel{\overline{\psi}}{\to} \overline{C^{(1)}}\to 0$ is a proper short exact sequence where $\overline{\phi}=\phi\mid_{\overline{A^{(1)}}}$ and $\overline{\psi}=\psi\mid_{\overline{B^{(1)}}}$. We introduce some results on $TFU$ extensions. Also, we establish conditions under which the $TFU$ extensions split.
\end{abstract}

\section{Introduction}

%------
\newcommand{\stk}{\stackrel}
Throughout, all groups are Hausdorff topological abelian groups and will be written additively. Let $\ell$ denote the category of locally compact abelian (LCA) groups with continuous homomorphisms as morphisms. A morphism is called proper if it is open onto its image, and a short exact sequence $0\to A\stackrel{\phi}{\to} B\stackrel{\psi}{\to}C\to 0$ in $\ell$ is said to be proper exact if $\phi$ and $\psi$ are proper morphisms. In this case the sequence is called an extension of $A$ by $C$ ( in $\ell $ ). Following \cite{FG1}, we let $Ext(C,A)$ denote the group of extensions of $A$ by $C$. The first Ulm subgroup of a group $G$ is denoted by $G^{(1)}$. In this paper, we introduce a new concept on extensions. An extension $0\to A \stackrel{\phi}{\to} B \stackrel{\psi}{\to} C\to 0$ in $\ell$ is called a $TFU$ extension if $0\to \overline{A^{(1)}}\stackrel{\overline{\phi}}{\to} \overline{B^{(1)}}\stackrel{\overline{\psi}}{\to} \overline{C^{(1)}}\to 0$ be an extension where $\overline{\phi}=\phi\mid_{\overline{A^{(1)}}}$ and $\overline{\psi}=\psi\mid_{\overline{B^{(1)}}}$. In Section 1, we show that every extension of a divisible group by an arbitrary LCA group is a $TFU$ extension (Corollary \ref{4}). A subgroup $H$ of a group $G$ is said to be pure if $H\bigcap nG=nH$ for every positive integer $n$ \cite{F}. An extension $0\to A \stackrel{\phi}{\to} B \stackrel{\psi}{\to} C\to 0$ is said to be a pure extension if $\phi(A)$ is pure in $B$ \cite{Fu1}. We show that every pure extension of a group by a compact totally disconnected group is a $TFU$ extension (Corollary \ref{6}). In Section 2, we establish some results on splitting of $TFU$ extensions (see Lemma \ref{10},\ref{11},\ref{12},\ref{16},\ref{13},).

The additive topological group of real numbers is denoted by $ \mathbb{R}$, $ \mathbb{Q}$ is the group of rationals with the discrete topology and $ \mathbb{Z}$, the group of integers. For any group $G$, $G_{0}$ is the identity component of $G$ and $tG$, the maximal torsion subgroup of $G$. For groups $G$ and $H$, $Hom(G,H)$ is the group of all continuous homomorphisms from $G$ to $H$, endowed with the compact-open topology. The dual group of $G$ is $\hat{G}=Hom(G, \mathbb{R}/ \mathbb{Z})$ and $(\hat{G},S)$ denoted the annihilator of $S\subseteq G$ in $\hat{G}$. For more on locally compact abelian groups, see \cite{HR}.

\section{$TFU$ extensions}
 Let $A,C\in \pounds$. In this section, we define the concept of a $TFU$ extensions of $A$ by $C$.

 \begin{definition}
An extension $0\to A \stackrel{\phi}{\to} B \stackrel{\psi}{\to} C\to 0$ in $\pounds$ is called a $TFU$ extension if $0\to \overline{A^{(1)}}\stackrel{\overline{\phi}}{\to} \overline{B^{(1)}}\stackrel{\overline{\psi}}{\to} \overline{C^{(1)}}\to 0$ is an extension.
\end{definition}

\begin{remark}\label{1}
Let $G_{1}$ and $G_{2}$ be two groups. An easy calculation shows that $(G_{1}\bigoplus G_{2})^{(1)}=G_{1}^{(1)}\bigoplus G_{2}^{(1)}$. So, $\overline{(G_{1}\bigoplus G_{2})^{(1)}}=\overline{G_{1}^{(1)}}\bigoplus \overline{G_{2}^{(1)}}$.
\end{remark}

\begin{lemma}\label{2}
For groups $A,C\in \ell$, the trivial extension $0\to A\to A\bigoplus C\to C\to 0$ is a $TFU$ extension.
\end{lemma}

\begin{proof}
It is clear by Remark \ref{1}.
\end{proof}

\begin{lemma}\label{3}
Let $G\in \ell$ and $H$ be a closed, divisible subgroup of $G$. Then, $\overline{(G/H)^{(1)}}=\overline{G^{(1)}}/H$.
\end{lemma}

\begin{proof}

Let $\pi:G\to G/H$ be the natural mapping. Then, $\pi(\overline{G^{1}})\subseteq \overline{\pi(G^{(1)})}=\overline{(G/H)^{(1)}}$. But, $\pi(\overline{G^{(1)}})=(\overline{G^{(1)}})/H$. Hence, $\overline{G^{(1)}}/H\subseteq \overline{(G/H)^{(1)}}$. Now, suppose that $x+H\in \overline{(G/H)^{(1)}}$. We show that $x\in \overline{G^{(1)}}$. Let $V$ be a neighborhood of $G$ containing $x$ and $n$, an arbitrary positive integer. Then, $y+H\in (V+H)/H \bigcap (G/H)^{(1)}\neq \phi$ for some $y\in G$. From $y+H\in (V+H)/H$, deduce that $y+H=z+H$ for some $z\in V$. Since $H$ is divisible, so $z=y+nh$ for some $h\in H$. On the other hand, $y+H\in (G/H)^{(1)}$. So, $y+H=ng+H$ for some $g\in G$. Hence, $y=ng+nh_{1}$ for some $h_{1}\in H$. Therefore, $z\in nG$. This shows that $z\in V\bigcap G^{(1)}$. Hence, $x\in \overline{G^{(1)}}$.
\end{proof}

\begin{corollary}\label{4}
Every extension of a divisible group by an arbitrary LCA group is a $TFU$ extension.
\end{corollary}

\begin{proof}
It is clear by Lemma \ref{3}.
\end{proof}

\begin{lemma}\label{5}
Let $G\in \ell$ and $H$ a closed, pure subgroup of $G$ such that $(G/H)^{(1)}=0$. Then, $G^{(1)}=H^{(1)}$.
\end{lemma}

\begin{proof}
Let $x\in G^{(1)}$ and $n$ an arbitrary positive integer. Then, $x=ng$ for some $g\in G$. So, $x+H\in n(G/H)$. Hence, $x+H\in (G/H)^{(1)}=0$. This shows that $G^{(1)}\subseteq H$. Since $H$ is pure, $G^{(1)}=H^{(1)}$.
\end{proof}

\begin{corollary}\label{6}
Every pure extension of a group by a compact totally disconnected group is a $TFU$ extension.
\end{corollary}

\begin{proof}
Let $G$ be a compact totally disconnected group. By Lemma \ref{5}, it is sufficient to show that $G^{(1)}=0$. By Theorem 24.26 of \cite{HR} , $\hat{G}$ is a torsion group. So, $G^{(1)}\subseteq(G,\hat{G})=0$ (see Theorem 24.22 of \cite {HR}).
\end{proof}

\begin{corollary}\label{14}
 Let $G\in \ell$ such that $nG$ is closed in $G$ for all positive integers $n$. Then, every pure extension of a group by $G/G^{(1)}$ is a $TFU$ extension.
\end{corollary}

\begin{proof}
 Let $x+G^{(1)}\in (G/G^{(1)})^{(1)}$ and $n$ be an arbitrary positive integer. Then, $x+G^{(1)}=ng+G^{(1)}$ for some $g\in G$. This shows that $x\in G^{(1)}$. Hence, $(G/G^{(1)})^{(1)}=0$.
\end{proof}

The dual of an extension $E:0\to A\to B\to C\to 0$ is defined by $\hat{E}:0\to \hat{C}\to \hat{B}\to \hat{A}$. The following example shows that the dual of a $TFU$ extension need not to be a $TFU$ extension.

\begin{example}
Consider the extension $E:0\to \mathbb {Z} \stackrel{\times 2}{\to} \mathbb {Z}\to \mathbb {Z}_{2}\to 0$. Clearly, $E$ is a $TFU$ extension. But, $\hat{E}:0\to \mathbb Z_{2}\to \mathbb R/\mathbb Z \stackrel{\times 2}{\to} \mathbb R/\mathbb Z \to 0$ is not a $TFU$ extension.
\end{example}

Recall that two extensions $0 \to A \stackrel{\phi_{1}}{\to} B \stackrel{\psi_{1}}{\to} C \to 0$ and $0 \to A \stackrel{\phi_{2}}{\to} X \stackrel{\psi_{2}}{\to} C \to 0$ are said to be equivalent if there is a topological isomorphism $\beta:B\to X$ such that the following diagram
\[
\xymatrix{
0 \ar[r] & A \ar^{\phi_{1}}[r] \ar^{1_{A}}[d] & B \ar^{\psi_{1}}[r] \ar^{\beta}[d] & C \ar[r] \ar^{1_{C}}[d]
& 0 \\
0 \ar[r] & A \ar^{\phi_{2}}[r] &   X \ar^{\psi_{2}}[r] & C \ar[r] & 0
}
\]
is commutative.

\begin{lemma}\label{7}
An extension equivalent to a $TFU$ extension is a $TFU$ extension.
\end{lemma}

\begin{proof}
Let
\begin{equation*}
E_1: 0 \to A \stackrel{\phi_1}{\to} B\stackrel{\psi_1}{\to} C \to 0
\end{equation*}
\begin{equation*}
E_2: 0 \to A \stackrel{\phi_2}{\to} X \stackrel{\psi_2}{\to} C \to 0
\end{equation*}
be two equivalent extensions such that $E_{1}$ is a $TFU$ extension. Then, there is a topological isomorphism $\beta:B\to X$ such that $\beta\phi_1=\phi_2$ and $\psi_2\beta=\psi_1$.
Since $\beta(\overline{B^{(1)}})=\overline{X^{(1)}}$, so
\begin{equation*}
\psi_{2}(\overline{X^{(1)}})=\psi_{2}\beta(\overline{B^{(1)}})=\psi_{1}(\overline{B^{(1)}})=\overline{C^{(1)}}
\end{equation*}
Hence, $\overline{\psi_{2}}:\overline{X^{(1)}}\to \overline{C^{(1)}}$ is surjective. Since $\phi_{2}=\beta\phi_{1}$ and $E_{1}$ is a $TFU$ extension, so $$\psi_{2}\phi_{2}(\overline{A^{(1)}})=\psi_{2}(\beta\phi_{1}(\overline{A^{(1)}}))=\psi_{1}\phi_{1}(\overline{A^{(1)}})=0$$ Hence, $Im \overline{\phi_{2}}\subseteq Ker \overline{\psi_{2}}$. Now, we show that $Ker \overline{\psi_{2}}\subseteq Im \overline{\phi_{2}}$. Let $x\in \overline{X^{(1)}}$ and $\psi_{2}(x)=0$. Then, there exists $b\in \overline{B^{(1)}}$ such that $x=\beta(b)$. Since $\psi_{1}(b)=\psi_{2}\beta(b)=\psi_{2}(x)=0$ and $E_{1}$ is a $TFU$ extension, so $b=\phi_{1}(a)$ for some $a\in \overline{A^{(1)}}$. Hence,
\begin{equation*}
\phi_{2}(a)=\beta\phi_{1}(a)=\beta(b)=x
\end{equation*}
and $0 \to \overline{A^{(1)}} \stackrel{\overline{\phi_2}}{\to} \overline{X^{(1)}} \stackrel{\overline{\psi_2}}{\to} \overline{C^{(1)}}\to 0 $ is an exact sequence. Since
\begin{equation*}
\overline{\psi_{2}}=\overline{\psi_{1}}(\overline{\beta})^{-1},\overline{\phi_{2}}=\overline{\beta}(\overline{\phi_{1}})
\end{equation*}
$\overline{\psi_{2}}$ and $\overline{\phi_{2}}$ are open. So, $E_{2}$ is a $TFU$ extension.
\end{proof}

\begin{lemma}\label{8}
 Let $C\in \ell$ be a torsion-free group, $0\to A \stackrel{\phi}{\to} B \stackrel{\psi}{\to} C\to 0$  be a $TFU$ extension and assume
\[
\xymatrix{
0 \ar[r] & A \ar^{\mu}[r] \ar^{1_{A}}[d] & X \ar^{\nu}[r] \ar^{\theta}[d] & Y \ar[r] \ar^{f}[d]
& 0 \\
0 \ar[r] & A \ar^{\phi}[r] & B \ar^{\psi}[r] & C \ar[r] & 0
}
\]

is the standard pullback diagram in $\pounds$ (See \cite{FG1}) such that $\overline{f}:\overline{Y^{(1)}}\to \overline{C^{(1)}}$ be a proper morphism. Then
\[
0 \to A \stackrel{\mu}{\to} X \stackrel{\nu}{\to} Y{\to} 0
\]
is a $TFU$ extension.
\end{lemma}

\begin{proof}
 We have
\[
X=\{(y,b)\in Y\bigoplus B: f(y)=\psi(b)\}.
\]
and
\[
\mu:a\mapsto (0,\phi(a)),  \nu:(y,b)\mapsto y,   \theta:(y,b)\mapsto b.
\]
Consider the following standard pullback diagram
\[
\xymatrix{
0 \ar[r] & \overline{A^{(1)}} \ar^{\phi'}[r] \ar^{}[d] & N \ar^{\psi'}[r] \ar^{}[d] & \overline{Y^{(1)}} \ar[r] \ar^{f}[d]
& 0 \\
0 \ar[r] & \overline{A^{(1)}} \ar^{\overline{\phi}}[r] & \overline{B^{(1)}} \ar^{\overline{\psi}}[r] & \overline{C^{(1)}} \ar[r] & 0
}
\]
where $N=\{(y,b)\in \overline{Y^{(1)}}\bigoplus \overline{B^{(1)}}: f(y)=\psi(b)\}.$ First, we show that $X^{(1)}=\{(y,b)\in Y^{(1)}\bigoplus B^{(1)}: f(y)=\psi(b)\}$. Let $(y,b)\in X^{(1)}$ and $n$ an arbitrary positive integer. Then, $y=ny_{1}$ and $b=nb_{1}$ for some $y_{1}\in Y$, $b_{1}\in B$. Also, $f(y_{1})=\psi(b_{1})$. We have
\begin{equation*}
f(y)=f(ny_{1})=\psi(nb_{1})=\psi(b)
\end{equation*}
Since $n$ be arbitrary, $(y,b)\in \{(y,b)\in Y^{(1)}\bigoplus B^{(1)}: f(y)=\psi(b)\}$. Conversely, let $(y,b)\in Y^{(1)}\bigoplus B^{(1)}$, $f(y)=\psi(b)$ and $n$ be an arbitrary positive integer. Then, $y=ny_{1}$ and $b=nb_{1}$ for some $y_{1}\in Y$, $b_{1}\in B$. Since $C$ is torsion-free, $n(f(y_{1})-\psi(b_{1}))=0$ deduce that $f(y_{1})=\psi(b_{1})$. Hence, $(y,b)\in nX$. Since $n$ be arbitrary, $(y,b)\in X^{(1)}$. An easy calculation shows that $\overline{X^{(1)}}=N$. Clearly, $\phi'=\overline{\mu}$ and $\psi'=\overline{\nu}$. Hence, $0 \to \overline{A^{(1)}} \stackrel{\overline{\mu}}{\to} \overline{X^{(1)}} \stackrel{\overline{\nu}}{\to} \overline{Y^{(1)}}{\to} 0$ is an extension.
\end{proof}

\begin{lemma}\label{9}
Let $A\in \ell$ be a divisible group and $C\in \pounds$. Then, a pushout of a $TFU$ extension of $A$ by $C$ is a $TFU$ extension.
\end{lemma}

\begin{proof}
Let $E:0\to A\stackrel{\phi}{\to} B\to C\to 0$ be a $TFU$ extension and $f:A\to G$ a proper morphism. Then, $fE:0\to G\to X\to C\to 0$ is a pushout of $E$, where $X=(G\bigoplus B)/H$ and $H=\{(-f(a),\phi(a));a\in A\}$ (See \cite{FG1}). Since $E$ is a $TFU$ extension, $E':0\to A\to \overline{B^{(1)}}\to \overline{C^{(1)}}\to 0$ is an extension. Hence, $hE'$ is an extension where $h:A\to \overline{G^{(1)}}$ defined by $h(a)=f(a)$ for every $a\in A$. But, $hE':0\to \overline{G^{(1)}}\to Y\to \overline{C^{(1)}}\to 0$ where $Y=(\overline{G^{(1)}}\bigoplus \overline{B^{(1)}})/K$ and $K=\{(-h(a),\phi(a));a\in A\}$. Clearly, $K=H$. Since $H$ is a closed, divisible subgroup of $G\bigoplus B$, so by Lemma \ref{3} and Remark \ref{1}, $\overline{X^{(1)}}=(\overline{G^{(1)}}\bigoplus \overline{B^{(1)}})/H=Y$. Hence, $fE$ is a $TFU$ extension.
\end{proof}

\section{Splitting of $TFU$ extensions}
An extension is called split if it is equivalent to the trivial extension. Let $Ext_{tfu}(C,A)$ be the class  of all equivalence classes of $TFU$ extensions of $A$ by $C$. Recall that for groups $A,C\in \ell$, $Ext(C,A)=0$ (or $Ext_{tfu}(C,A)=0$) deduce that every extension (or $TFU$ extension) of $A$ by $C$ splits. In this section, we establish  some conditions on $A$ and $C$ such that $Ext_{tfu}(C,A)=0$.

\begin{lemma}\label{10}
Let $A$ be a discrete group such that $Ext_{tfu}(X,A)=0$ for all groups $X\in \ell$. Then, $A$ is a reduced group such that $A/tA$ is a divisible group.
\end{lemma}

\begin{proof}
Let $A$ be a discrete group such that $Ext_{tfu}(X,A)=0$ for all groups $X\in \pounds$. Let $D$ be a divisible subgroup of $A$ and $C$ a connected group. By Corollary 2.10 of \cite{FG2}, we have the following exact sequence
\begin{equation*}
Hom(C,A/D)\to Ext(C ,D)\stackrel{i_{\ast}}{\to}Ext(C,A)
\end{equation*}
 Since $C$ is a connected group and $A/D$ a discrete group, $Hom(C,A/D)=0$. Hence, $i_{\ast}$ is injective. By Lemma \ref{9}, $i_{\ast}(Ext_{tfu}(C ,D))\subseteq Ext_{tcf}(C,A)=0$. So, $Ext_{tcf}(C,D)=0$. By Corollary \ref{4}, $Ext(C,D)=0$. Hence, $D=0$ (see Theorem 3.3 of \cite{FG2}). So, $A$ is a reduced group. Now, we show that $A/tA$ is a divisible group. By Corollary \ref{6}, $Pext(\hat{( \mathbb{Q}/ \mathbb{Z})},A)\subseteq Ext_{tfu}(\hat{( \mathbb{Q}/ \mathbb{Z})},A)=0$. So, $Ext(\hat{( \mathbb{Q}/ \mathbb{Z})},A/tA)=0$. Hence, $A/tA$ is a divisible group (see the proof of Theorem 2 of \cite{AS2} ).
\end{proof}

\begin{lemma}\label{11}
Let $G$ be a compact group such that $Ext_{tfu}(X,G)=0$ for all groups $X\in \ell$. Then, $G\cong \bigoplus ( \mathbb{R}/  \mathbb{Z})^{\sigma}\bigoplus  \mathbb{R}^{n}\bigoplus M$, where $n$ is a positive integer, $\sigma$ a cardinal number and $M$ is a direct product of finite cyclic groups.
\end{lemma}

\begin{proof}
Let $G$ be a compact group and $Ext_{tfu}(X,G)=0$ for all groups $X\in \ell$. Then, $Ext_{tfu}(C,G)=0$ for all connected groups $C\in \ell$. Consider the exact sequence $0\to G_{0} \stackrel{i}{\to} G \to G/G_{0} \to 0$. We have the following exact sequence
\begin{equation*}
0=Hom(C,G/G_{0})\to Ext(C ,G_{0})\stackrel{i_{\ast}}{\to}Ext(C,G)
\end{equation*}
Hence, $i_{\ast}$ is injective. By Lemma \ref{9}, $i_{\ast}(Ext_{tfu}(C ,G_{0}))\subseteq Ext_{tcf}(C,G)=0$. So, $Ext_{tcf}(C,G_{0})=0$. By Corollary \ref{4}, $Ext(C,G_{0})=0$.
Hence, $G_{0}\cong( \mathbb{R}/ \mathbb{Z})^{\sigma}\bigoplus  \mathbb{R}^{n}$ (see Theorem 3.3 of \cite{FG2}). By Corollary 3.4 of \cite{FG2}, $G_{0}$ splits in $G$. So, $G\cong G_{0}\bigoplus G/G_{0}$. Set $M=G/G_{0}$. Then, $M$ is a compact totally disconnected group. By Corollary \ref{6}, $Pext(X,G)\subseteq Ext_{tfu}(X,G)=0$ for all compact totally disconnected groups $X$. So, $Pext(X,M)=0$ for all compact totally disconnected groups $X$. Let $Y$ be a compact group. By Proposition 2 of \cite{Fu1}, we have the following exact sequence
\begin{equation*}
 0=Pext(Y/Y_{0},M)\to Pext(Y,M)\to Pext(Y_{0},M)
\end{equation*}
By Theorem 4.2 of \cite{AS} , $Pext(Y_{0},M)=0$. Hence, $Pext(Y,M)=0$ for all compact groups $Y\in\ell$. By Proposition 53.4 of \cite{F} and Lemma 2.3 of \cite{L2}, $\hat{M}$ is a direct sum of finite cyclic groups. Hence, $M$ is a direct product of finite cyclic groups.
\end{proof}

\begin{lemma}\label{12}
Let $A$ be a $\sigma-$compact group such that $A^{(1)}=0$ and $C$, a divisible group in $\ell$. Then, every $TFU$ pure extension of $A$ by $C$ splits.
\end{lemma}

\begin{proof}
Let $E:0\to A \stackrel{\phi}{\to} B \stackrel{\psi}{\to} C\to 0$ be a $TFU$ pure extension. An easy calculation shows that $B=\overline{B^{(1)}}+\phi(A)$. Since $\phi(A)$ is pure in $B$, $\overline{B^{(1)}}\bigcap \phi(A)=0$. Hence, by Corollary 3.2 of \cite{FG1}, $B=\overline{B^{(1)}}\bigoplus \phi(A)$. So, $E$ splits.
\end{proof}

\begin{lemma}\label{15}
Let $A$ be a discrete, torsion-free group. Then, $A\cong A^{(1)}\bigoplus A/A^{(1)}$.
\end{lemma}

\begin{proof}
Let $A$ be a discrete, torsion-free group. First, we show that $A/A^{(1)}$ is torsion-free. Let $m(a+A^{(1)})=0$ for some positive integer $m$. Then, $ma\in A^{(1)}$. Hence, $ma=mna_{1}$ for an arbitrary positive integer $n$. So, $a=na_{1}$. Since $n$ be arbitrary, $a\in A^{(1)}$. So, $A^{(1)}$ is pure in $A$. On the other hand, $A^{(1)}\subseteq nA$ for every positive integer $n$. Hence, $A^{(1)}$ is a divisible group. By Theorem 21.1 of \cite{F}, $A\cong A^{(1)}\bigoplus A/A^{(1)}$.
\end{proof}

\begin{lemma}\label{16}
Let $A$ be a discrete, torsion-free group such that $Ext_{tfu}(A,X)=0$ for all groups $X\in \ell$. Then, $A\cong (\bigoplus_{\sigma}  \mathbb{Z})\bigoplus D$ where $D$ is a discrete, torsion-free divisible group.
\end{lemma}

\begin{proof}
Let $A$ be a discrete, torsion-free group such that $Ext_{tfu}(A,X)=0$ for all groups $X\in \ell$. By Lemma \ref{15}, $A\cong A^{(1)}\bigoplus A/A^{(1)}$. By Lemma \ref{8}, $Ext_{tfu}(A/A^{(1)},X)=0$ for all groups $X\in \ell$. Hence, $Ext(A/A^{(1)},X)=0$ for all groups $X\in \pounds$ (see Corollary \ref{14}). By Theorem 3.3 of \cite{M}, $A/A^{(1)}\cong \bigoplus_{\sigma} \mathbb{Z}$.
\end{proof}

\begin{lemma}\label{13}
Let $G$ be a compact, torsion-free group. Then, $Ext_{tfu}(G,X)=0$ for all groups $X\in \ell$ if and only if $G=0$.
\end{lemma}

\begin{proof}
Let $G$ be a compact, torsion-free group and $Ext_{tfu}(G,X)=0$ for all groups $X\in \ell$. Let $X$ be a totally disconnected group in $\ell$. Consider the following exact sequence
\begin{equation*}
0=Hom(G_{0},X)\to Ext(G/G_{0} ,X)\stackrel{\pi_{\ast}}{\to}Ext(G,X)
\end{equation*}
By Lemma \ref{8}, $\pi_{\ast}(Ext_{tfu}(G/G_{0} ,X))\subseteq Ext_{tfu}(G,X)=0$. So, $Ext_{tfu}(G/G_{0} ,X)=0$. By Corollary \ref{6}, $Ext(G/G_{0} ,X)=0$. Hence, $G=G_{0}$ (see Theorem 3.5 of \cite{FG2} ). Since $ \mathbb{Q}$ is divisible, $Ext(G, \mathbb{Q})=0$ which is a contradiction by Lemma 2.9 of \cite{A}. So, $G=0$.
\end{proof}


\begin{thebibliography}{5}
\bibitem{A} A. A. Alijani, On generalized $\pounds-$ cotorsion LCA groups, arXiv:2210.04178[math.GR].

\bibitem{AS} A. A. Alijani and H. Sahleh, On t-extensions of abelian groups, {\it Khayyam J. Math}, 1 (5) (2019), 78-86.

\bibitem{F}  L. Fuchs, {\it Infinite Abelian Groups}, Vol. 1, Academic Press, (1970).

\bibitem{Fu1}  R. O. Fulp, Homological study of purity in locally compact groups, {\it Proc. London Math. Soc}, 21 (1970), 501-512.


\bibitem{FG1}  R.O. Fulp and P. Griffith, Extensions of locally compact abelian groups I, {\it Trans. Amer. Math. Soc}, 154 (1971),341-356.

\bibitem{FG2} R. O. Fulp and P. Griffith, Extensions of locally compact abelian groups II, {\it Trans. Amer. Math. Soc}, 154 (1971), 357-363.

\bibitem{HR}  E. Hewitt and K. Ross, {\it Abstract Harmonic Analysis}, vol. 1, 2nd edn, Springer-Verlog, Berlin (1979).

\bibitem{L2}  P. Loth, Pure extensions of locally compact abelian groups, {\it Rend. Sem. Mat. Univ. Padova},  116 (2006), 31-40.

\bibitem{AS2}  H. Sahleh and A. A. Alijani, The pure injectives and pure projectives in the category of totally disconnected, locally compact abelian groups, Math. Sci. 7(49) (2013).
\end{thebibliography}
\end{document}